\documentclass[letter, 12pt,reqno]{amsart}
\usepackage[utf8]{inputenc}
\usepackage{amsfonts,amsmath,mathtools,amssymb,xcolor,amsthm,stackrel}

\usepackage{mathrsfs}

\usepackage[left=1in,right=1in,top=1in,bottom=1in]{geometry}

\usepackage[hidelinks]{hyperref}

\usepackage{amsaddr}

\newtheorem{theorem}{Theorem}
\newtheorem{lemma}[theorem]{Lemma}
\newtheorem{remark}[theorem]{Remark}

\newtheorem{proposition}[theorem]{Proposition}

\title{The elastic ray transform}
\author{Joonas Ilmavirta}
\address{Department of Mathematics and Statistics,University of Jyv\"askyl\"a, Jyv\"askyl\"a,  Finland ({\tt{joonas.ilmavirta@jyu.fi}})}

\author{Antti Kykk\"anen}
\address{Department of Computational Applied Mathematics and Operations Research, Rice University, Houston, TX, USA ({\tt{antti.kykkanen@rice.edu}}), Corresponding author}

\author{Teemu Saksala}
\address{Department of Mathematics, NC State University, Raleigh, NC, USA  
({\tt{tssaksal@ncsu.edu}})}
\date{\today}

\usepackage[finnish,english]{babel}

\newcommand{\der}{\mathrm{d}}

\newcommand{\R}{{\mathbb R}}
\newcommand{\C}{{\mathbb C}}

\newcommand{\abs}[1]{\left\lvert#1\right\rvert}
\newcommand{\norm}[1]{\left\Vert#1\right\Vert}
\newcommand{\ip}[2]{\left\langle#1,#2\right\rangle}
\newcommand{\iip}[2]{\left(#1,#2\right)}

\newcommand{\dummy}{{\,\cdot\,}}

\newcommand{\eps}{\varepsilon}
\newcommand{\stiffnesstensor}[1]{\etensor{#1}{2}}
\newcommand{\etensor}[2]{E^{#2}(#1)}
\newcommand{\schwartz}{\mathscr{S}}
\newcommand{\sphere}[1]{{\mathbb S}^{#1}}

\DeclareMathOperator{\im}{im}
\newcommand{\elasticsymmetrization}{\varepsilon}

\newcommand{\eXRT}[1]{X^{#1}}
\newcommand{\XRT}{X}
\newcommand{\eHesse}{H}
\newcommand{\eGrad}{K}
\newcommand{\idmatrix}{I}
\newcommand{\sotimes}{\otimes_{\mathrm{s}}}

\newcommand{\polarizationset}[2]{Q_{#2}(#1)}

\newcommand{\hsymb}[1]{\widehat{H}_{#1}}
\newcommand{\htsymb}[1]{\widehat{H}^\ast_{#1}}
\newcommand{\ksymb}[1]{\widehat{K}_{#1}}
\newcommand{\ktsymb}[1]{\widehat{K}^\ast_{#1}}

\newcommand{\sol}{\mathrm{Sol}}
\newcommand{\pot}{\mathrm{Pot}}

\newcommand{\kuvak}[1]{A_{#1}}
\newcommand{\kuvai}[1]{B_{#1}}
\newcommand{\ydinki}[1]{C_{#1}}

\newcommand{\kinverse}[1]{\widehat{K}^{-1}_{#1}}
\newcommand{\iinverse}[1]{\widehat{H}^{-1}_{#1}}

\newcommand{\avaruus}{A}

\newcommand{\smooth}[1]{C^\infty(\R^n;\etensor{n}{#1})}
\newcommand{\ltwo}[1]{L^2(\R^n;\etensor{n}{#1})}

\begin{document}

\begin{abstract}
We introduce and study a new family of tensor tomography problems. At rank 2 it corresponds to linearization of travel time of elastic waves, measured for all polarizations. We provide a kernel characterization for ranks up to 2. The kernels consist of potential tensors, but in an unusual sense: the associated differential operators have degree 2 instead of the familiar 1. The proofs are based on Fourier analysis, Helmholtz decompositions, and cohomology.
\end{abstract}

\maketitle

\section{Introduction}
\label{sec:introduction}

We introduce and study ray transforms of specific kinds of even rank tensors over Euclidean spaces.
The tensor fields are not fully symmetric and the ray transform is not of any of the usual types that are concerned in the current literature.
We say that an elastic $m$-tensor in $\R^n$ is a symmetric rank $2m$ tensor over the symmetric square of $\R^n$.
In the case $m=2$ this is exactly the symmetry of the stiffness tensor in elasticity --- whence the name.

For each $m \in \{0,1,2,\ldots\}$ there is a natural concept of an elastic X-ray transform, and
in the case $m=2$ this arises from
the linearization of travel time for elastic waves about a homogeneous and isotropic background medium.
This is analogous to how the usual X-ray transform of symmetric $2$-tensors arises from the linearization of travel times for waves corresponding to a Riemannian metric \cite[Section 1.1]{Sharafutdinov1994}.
Symmetric $2$-tensors have the same symmetries as Riemannian metrics, whereas elastic $2$-tensors have those of a stiffness tensor.

This travel time linearization is the reason we set out to study the elastic ray transform.
We only linearize travel time at a homogeneous and isotropic background.
Homogeneity makes most of our operators simply Fourier multipliers and ensures that seismic waves in the unperturbed medium travel along straight lines.
Isotropy is less crucial, but simplifies matters considerably by making the Fourier multipliers rotation invariant.
The perturbation can be arbitrarily inhomogeneous and anisotropic.

While the physically relevant case is $m=2$, we set up the tensor tomography problem for all $m \in \{0,1,2,\ldots\}$.
We provide the complete characterization of the kernel of these transforms in the cases $m=0,1,2$.
The kernel is given by elastic tensor fields that are of potential form, but in a very different sense than for the usual tensor tomography problem.
For example, for $m=1$ the relevant differential operator is of degree $2$ and for $m=2$ the kernel consists of images of two partial differential operators of degrees 1 and 2 respectively.

For technical reasons one of our theorems have to assume $n\geq3$, but fortunately the physical setup has indeed dimension $n=3$.

\subsection{The elastic X-ray transform}

We denote by $\schwartz(\R^n;E)$ the Fréchet space of Schwartz functions taking values in a finite-dimensional complex vector space $E$. We reserve the notation $E\sotimes F$ for the symmetric tensor product of complex vector spaces $E$ and $F$, and 
denote by $E^{\sotimes m}$ the $m$th symmetric tensor power of a vector space $E$.
For any dimension $n\geq1$ and any rank $m\geq0$ we define the space of elastic $m$-tensors as
\begin{equation}
\etensor{n}{m}
=
(\R^n\sotimes\R^n)^{\sotimes m}
.
\end{equation}
In the case $m\geq 2$ the space $\etensor{n}{m}$ is not the same as the space $(\R^n)^{\sotimes 2m}$ of fully symmetric tensors because the symmetries are different:
\begin{equation}
(\R^n)^{\sotimes 2m}
\subsetneq
(\R^n\sotimes\R^n)^{\sotimes m}
\subsetneq
\R^{\otimes 2m}.
\end{equation}
For a concrete example, take $m=2$.
For all $a\in\etensor{n}{2}$ it is true that $a_{ijkl}=a_{jikl}=a_{klij}$, but typically $a_{ijkl}\neq a_{ikjl}$.

We denote the associated elastic symmetrization operator, taking a tensor without symmetries to a tensor with elastic symmetries, by
\begin{equation}
\elasticsymmetrization
\colon
(\R^n)^{\otimes 2m}
\to
\etensor{n}{m}
.
\end{equation}

For $v\in\sphere{n-1}$ we use the notation
\begin{equation}
\polarizationset{v}{}
=
\R v
\cup
v^\perp
\subset
\R^n
,
\end{equation}
which is a union of linear subspaces of dimension one and codimension one respectively.
We call $v \in \sphere{n-1}$ a direction, $q\in\polarizationset{v}{}$ a polarization, and $f\in\schwartz(\R^n;\etensor{n}{m})$ or $f \in \ltwo{m}$ an elastic $m$-tensor field.

We define the elastic ray transform of an elastic tensor field as the line integral of the tensor field contracted against a tensor power of the product of velocity and polarization.
Specifically, for all $f\in\schwartz(\R^n;\etensor{n}{m})$, $x\in\R^n$, $v\in\sphere{n-1}$, and $q\in\polarizationset{v}{}$ we denote
\begin{equation}
\label{eqn:exrt}
\eXRT{m}_{v,q}f(x)
=
\int_\R
\ip{f(x+tv)}{(v\otimes q)^{\otimes m}}
\der t
.
\end{equation}
For example, in the case $m=2$ we have the following component expression
\begin{equation}
\eXRT{2}_{v,q}f(x)
=
\int_\R
\sum_{i,j,k,l}
f_{ijkl}(x+tv)v_iq_jv_kq_l
\der t
.
\end{equation}
Formula \eqref{eqn:exrt} defines a family $\eXRT{m} \coloneqq \{\eXRT{m}_{v,q}\}$ of operators which we call the elastic X-ray transform.

We say that the elastic X-ray transform of $f$ vanishes and write $\eXRT{m}f=0$ if $\eXRT{m}_{v,q}f = 0$ for all $v \in \sphere{n-1}$ and $q \in \polarizationset{v}{}$.
Hence, we adopt the notation 
\begin{equation}
\ker(\eXRT{m}):=\bigcap_{v\in\R^n,\: q\in\polarizationset{v}{}}\ker(\eXRT{m}_{v,q}).
\end{equation}
If $f \in \schwartz(\R^n;\etensor{n}{m})$, it is straightforward to verify that for all $v \in \sphere{n-1}$ and unit vectors $q \in \polarizationset{v}{}$ have that
\begin{equation}
\norm{\eXRT{m}_{v,q}f}_{L^2(v^\perp;\R)}
\leq
\norm{f}_{\ltwo{m}}.
\end{equation}
Therefore, each $\eXRT{m}_{v,q}$ extends to a continuous linear map from $\ltwo{m}$ to $L^2(v^\perp;\R)$.

Our main results are kernel characterizations of the elastic X-ray transforms of fields of rank $1$ and $2$. We give an overview of the results in the next subsections. 

\subsection{General rank}

Let $D$ be the operator acting on an $m$-tensor field $f$ by
\begin{equation}
(Df)_{ji_1 \cdots i_m}
=
\partial_jf_{i_1\cdots i_m}
\end{equation}
and let $D^2f = D(Df)$.
For any $m\geq1$, we define a differential operator $\eHesse_m$ mapping $\etensor{n}{m-1}$-fields to $\etensor{n}{m}$-fields by
\begin{equation}
\label{eq:Op_Hesse}
\eHesse_m h
=
\elasticsymmetrization(D^2h)
.
\end{equation}

This is a well defined map e.g. $
\eHesse_m
\colon
\schwartz(\R^n;\etensor{n}{m-1})
\to
\schwartz(\R^n;\etensor{n}{m})
$.
For all ranks $m \geq 1$ the image of $H_m$ lies always in the kernel of the elastic X-ray transform $X^m$.
See Proposition \ref{prop:hesse-in-kernel-all-ranks} for details.

We do not pursue full kernel characterization for $m\geq3$.
The two main reasons are that the known physical situation is $m=2$ and explicit calculations --- especially computer-assisted ones --- are not possible for general rank.

\subsection{Rank 0}

Rank $0$ elastic tensor fields are scalar fields (functions) on $\R^n$.
The elastic ray transform, $\eXRT{0}_{v,q}$ for any $v \in \sphere{n-1}$ and $q \in \polarizationset{v}{}$, is the classical X-ray transform of functions on $\R^n$. This transform is well known to be injective~\cite{Helgason2011,Radon1917}, so $\ker(\eXRT{0})=0$.

\subsection{Rank 1}

For elastic tensor fields on rank $1$ we have the following kernel characterization.

\begin{theorem}[Proven in Section~\ref{sec:proof-for-rank-1}]
\label{thm:ker-characterization-1}
Let $n\geq2$. A tensor field $f\in C_c^\infty(\R^n;\etensor{n}{1})$ satisfies $\eXRT{1}f=0$ if and only if $f=\eHesse_1 h$ for some $h\in C_c^\infty(\R^n)$.
\end{theorem}

It is noteworthy that $\eHesse_1$ is a second order differential operator unlike the first-order symmetric covariant derivative familiar from longitudinal tensor tomography.
The proof of Theorem~\ref{thm:ker-characterization-1} is based on a cohomological argument and symmetry properties of elastic tensor fields. The proof sheds light on why the kernel is characterized by a second order differential operator.

\subsection{Rank 2}

For elastic tensor fields of rank $2$ we prove injectivity of~$\eXRT{2}$ up to a gauge. In what follows, $\eHesse_2$ is the second order operator introduced in equation~\eqref{eq:Op_Hesse} and $\eGrad$ is a first order operator mapping vector fields to tensor fields $\R^n \to \etensor{n}{2}$ as
\begin{equation}
\label{eqn:oper-k}
\eGrad W = \elasticsymmetrization(DW \otimes I)
\end{equation}
where $I$ is the identity matrix. This map can be defined between various function spaces. Our main result states that the images of the operators $\eGrad$ and $\eHesse$ always lie in the kernel of the elastic ray transform, while those fields that are in the intersection of the kernels of the respective adjoint operators $\eGrad^\ast$ and $\eHesse^\ast$ are never annihilated by the elastic ray transform. This is summarized in the following theorem.
In Theorem~\ref{thm:ims-and-kers-2} all operators operate between the Schwartz classes.

\begin{theorem}[Proven in Section~\ref{sec:proof-for-rank-2}]
\label{thm:ims-and-kers-2}
For $n\geq1$ the following hold:
\begin{enumerate}
\item\label{claim:ker1}
$\im(\eHesse)+\im(\eGrad)\subset\ker(\eXRT{2})$.
\item\label{claim:ker2}
$\ker(\eHesse^*)\cap\ker(\eGrad^*)\cap\ker(\eXRT{2})=0$.
\end{enumerate}
\end{theorem}

We supplement the result of Theorem~\ref{thm:ims-and-kers-2} by proving that elastic $2$-tensor fields admit Helmholtz decompositions into a potential part and a solenoidal part in Section~\ref{sec:helmholtz}. For technical reasons the existence of such decompositions is proved only in dimensions $n \geq 3$. In particular, this assumption is necessary to prove Lemma~\ref{lma:schwartz-decomposition}.
We believe that the dimensional assumption is merely an artifact of the proof method, not a feature of the problem.
With a Helmholtz type decomposition adapted to the operators~$K$ and~$H$ (instead of the gradient) we get the kernel characterization in Theorem~\ref{thm:ker-characterization-2}. Alas, we do not have a physical interpretation for the kernel of $\eXRT{2}$.

\begin{theorem}[Proven in Section~\ref{sec:helmholt-final-proof}]
\label{thm:ker-characterization-2}
Let $n \geq 3$ and let $f \in \ltwo{2}$. Then $Xf = 0$ if and only if $f =KW + Hh$ for some $W \in H^1_K$ and $h \in H^2_H$.
\end{theorem}

Here $H^1_K$ is the space of vector fields for which $\eGrad W$ is defined in the weak sense and belongs to  $L^2$, and $H^2_H$ is the space of elastic $1$-tensor fields for which $\eHesse_2 h$ is defined in the weak sense and is an element of $L^2$ (see Section~\ref{sec:helmholt-final-proof} for the detailed definitions).

For rank $2$ elastic tensors it is an important feature of the problem that the polarization set is exactly $\polarizationset{v}{} = \R v \cup v^\perp$ and not for example the vector space $\R^n$ generated by the set $\polarizationset{v}{}$.
The definition of the ray transform and the differential operators make perfect sense if we replace the polarization set $\polarizationset{v}{}$ with $\R^n$, but this changes the problem drastically. In this case it no longer holds that $\im(\eGrad) \subseteq \ker(\eXRT{2})$ since for $q \in \R^n$ such that $q \ne v$ and $q \not\perp v$ we do not get the total derivative in~\eqref{eqn:k-in-kernel}.

\subsection{Organization of the paper}

In Section~\ref{sec:rationale-comparison} we provide the broader motivation to the study of the elastic ray transform (Section~\ref{sec:linearization}). Also, we establish connections to the existing literature on related ray transforms, such as the longitudinal and the mixed ray transforms (Section~\ref{sec:other-transforms}).

Section~\ref{sec:ranks} is divided into subsections according to the rank of the elastic tensor field under consideration. In Section~\ref{sec:general-rank-proof} we prove that for any rank $m$ the image of $\eHesse_m$ is always contained in the kernel of the elastic ray transform $\eXRT{m}$. Section~\ref{sec:proof-for-rank-1} is concerned with elastic tensor fields of rank $1$. In here, we prove Theorem~\ref{thm:ker-characterization-1} using the Poincar\'e lemma. Additionally, in Theorem~\ref{thm:ims-and-kers-1}, we show by different methods that the ray transform $\eXRT{1}$ is solenoidally injective in analogy with our main result Theorem~\ref{thm:ims-and-kers-2}. Theorem~\ref{thm:ims-and-kers-2} is proved in Section~\ref{sec:proof-for-rank-2} with the aid of lemmas~\ref{lma:fst},~\ref{lma:no-pw-kernel} and~\ref{lma:reduction}. Lemma~\ref{lma:fst} is an adaptation of the Fourier slice theorem. Lemmas~\ref{lma:no-pw-kernel} and~\ref{lma:reduction} are rather technical requiring computer assisted proofs and their proofs are postponed until Section~\ref{sec:cas-proofs}.

Section~\ref{sec:helmholtz} is devoted to stating and proving Helmholtz decompositions, see Theorem \ref{thm:helmholtz}, for elastic tensor fields of rank 2 in dimensions 3 and higher. The section consists of three parts. In Section~\ref{sec:pointwise-decompositions} we prove by algebraic arguments a pointwise decomposition of elastic tensors on the Fourier domain. In Section~\ref{sec:global-decompositions-Fourier} we show that the pointwise algebraic decompositions can be promoted to decompositions of elastic tensor fields of certain regularity classes, having particular emphasis on the $L^2$-regular tensor fields. The main content is the regularity and decay analysis in Lemma~\ref{lma:schwartz-decomposition}. Section~\ref{sec:helmholt-final-proof}, finalizes the proof of the Helmholtz decomposition using density arguments in Theorem~\ref{thm:helmholtz} which is then used in combination with Theorem~\ref{thm:ims-and-kers-2} to prove Theorem~\ref{thm:ker-characterization-2}.

Section~\ref{sec:cas-proofs} completes the proofs of Lemmas~\ref{lma:no-pw-kernel} and~\ref{lma:reduction} and explains the Maxima codes used in the proofs.

\subsection*{Acknowledgements}
J.I. and A.K. were supported by the Research Council of Finland (Flagship of Advanced Mathematics for Sensing Imaging and Modelling grant 359208; Centre of Excellence of Inverse Modelling and Imaging grant 353092; and other grants 351665, 351656, 358047).
J.I. also acknowledges support from the V\"ais\"al\"a project grant by the Finnish Academy of Science and Letters.
A.K. was supported by the Geo-Mathematical Imaging Group at Rice University.
T.S. was supported by the National Science Foundation grant DMS-2204997.
We thank Maarten V. de Hoop for discussions and hosting us on our many visits to Rice University. We thank the anonymous referees for the valuable remarks and suggestions. These improved the article considerably.

\section{Rationale and comparison to other problems}
\label{sec:rationale-comparison}

\subsection{Linearized travel time and the elastic wave equation}
\label{sec:linearization}

This subsection provides a somewhat soft justification for our ray transform, arising from the linearization of travel times of elastic waves.

Consider a matrix-valued hyperbolic operator $\square_A=\partial_t^2-\Delta_A$ depending on some tensor field $A$.
At least if $\square_A$ is of real principal type, the singularities propagate by the Hamiltonian flows with the Hamiltonians given by the different eigenvalues $\lambda_k(\sigma(\Delta_A))$ of the spatial component of the principal symbol $\sigma(\Delta_A)$ (c.f. \cite{dencker1982propagation}).
We model wave propagation by these flows even though singularities might not follow them in full generality.
We say that the travel time of a wave, associated with the eigenvalue $\lambda_k$, between two points is the time it takes for the respective Hamiltonian flow to travel between the points.   

Let us consider a one-parameter family $s\mapsto A_s$ of ``material parameters'' and differentiate the travel time of the associated waves with respect to $s$.
The general principle is that this derivative is proportional to the integral of $\partial_s\lambda_k(\sigma(\Delta_{A_s}))|_{s=0}$ over the integral curve of the unperturbed trajectory at $s=0$.

When the $k$th eigenvalue $\lambda_k(\sigma(\Delta_{A_0}))$ degenerates, the travel time derivative is more complicated.
To reduce clutter, we may consider $s\mapsto M_s\coloneqq \sigma(\Delta_{A_s})$ simply as a smooth one-parameter family of symmetric real matrices.
If the $k$th eigenvalue $\lambda_k(M_0)$ is degenerate, then the perturbation may split it and the first order corrections to the eigenvalue are the eigenvalues of block of the perturbation matrix $M'\coloneqq\partial_sM_s|_{s=0}$ corresponding to the eigenspace of $\lambda_k(M_0)$.
The map from $M'$ to these eigenvalues is not linear.
For more details on the spectral perturbation theory of linear operators, see Appendix~\ref{app:spt} or e.g.~\cite{Kato1966}.

Our interest is in such travel time linearization for the elastic wave equation: $(\partial_t^2 - L)u(t,x)=0$ which can be written in Cartesian coordinates as
\begin{equation}
Lv(x)_i
=
\rho(x)^{-1}
\sum_{j,k,l}
\partial_{x_j}(c_{ijkl}(x)\partial_{x_l}v_k(x))
.
\end{equation}
The stiffness tensor field $c$ takes values in $\etensor{n}{2}$ and the density $\rho$ is a scalar.
We assume that the background stiffness tensor field at $s=0$ is isotropic and homogeneous; we will discuss in a moment the simplifications afforded by this assumption.
There are two eigenvalues of the principal symbol of the spatial part (also known as the Christoffel matrix $\Gamma_{ik}(x,v)=\rho^{-1}(x)\sum_{j,l}c_{ijkl}(x)v_jv_l$), one for pressure waves (polarized along the wave, multiplicity $1$) and one for shear waves (polarized orthogonal to the wave, multiplicity $n-1$).
The aim is to recover as much information as possible of the perturbation $f$ of the density normalized stiffness tensor $\rho^{-1}c$ field from the linearized travel time data.

The linearization of pressure wave travel times gives our elastic ray transform $\eXRT{2}$ for $q=v$.
This is exactly the longitudinal ray transform.
The eigenvalue is simple.

For shear waves the situation is more complicated when $n\geq3$ because for some points $(x,v)\in \R^n \times \sphere{n-1}$ some of the eigenvalues of $\Gamma(x,v)$ are degenerate. 
If the perturbation is homogeneous, the situation reduces to spectral perturbation theory as described above.
The polarization $q$ is an eigenvector of the shear block of the perturbation of the Christoffel matrix.
The admissible set of polarizations depends on the data and the ``linearization''\footnote{It is a general feature of spectral perturbation theory that Gateaux derivative is well defined but does not correspond to a Fr\'echet derivative. This renders the linearized problem non-linear. See Appendix~\ref{app:spt}.} does not depend linearly on the perturbation $f$.
To turn the linearized problem linear, we allow $q$ to take all values orthogonal to $v$ and do not allow $q$ to vary in time in a way that depends on $f$.
Therefore our $\eXRT{2}$ is truly the linearization of the travel time data only in two dimensions.
We are not aware of a full and rigorous treatment of degeneration-breaking linearization of travel time of elastic waves.

Our data for pressure waves corresponds to the longitudinal ray transform and for shear waves a mixed ray transform.
The natural symmetry types of tensors for these two problems are different, and both also different from the natural symmetry of stiffness tensors.
To stay relevant to the physical problem, the symmetry we assume is the one stemming from the elastic model.

We made two assumptions on the background stiffness tensor: homogeneity and isotropy.
In many anisotropic situations it happens that the degeneracy of the principal symbol (the Christoffel matrix) depends on direction.
This makes the set $Q(v)$ and even its dimension depend on $v$, making the behavior of the ray transform far less clean.
Homogeneity has the advantage that the trajectories we integrate over are simply lines $t\mapsto x+tv$.
The constant velocity $v$ need not be unit length and will be different for pressure and shear waves, but this can be scaled out of the ray transform.
The same is true for anisotropic homogeneous media, where the set of admissible speeds also depends on direction.
The relevant ray transform does not depend on the constant Lam\'e parameters of the background model, whereas for an inhomogeneous background the dependence would be significant.
Translation invariance also serves to make Fourier analysis accessible.
Getting rid of these assumptions is an obvious direction for future developments.

\subsection{Relation other kinds of tensor tomography}
\label{sec:other-transforms}

Even thought the elastic ray transform (for $m=2$) combines the data of the longitudinal and the mixed ray transform the symmetry type of the tensor field is different. For this reason we cannot simply combine results on the longitudinal and mixed transforms to prove results for the elastic ray transform.

The longitudinal X-ray transform of a fully symmetric $m$-tensor field is defined as
\begin{equation}
X_{L}f(x,v)
=
\int_\R
\ip{f(x+tv)}{v^{\otimes m}}
\,dt.
\end{equation}
This transform is known to be solenoidally injective on various different function spaces.
This transform has been extensively studied in many geometries~\cite{IM2019,PSU2014,PSU2023,Sharafutdinov1994}.
The Helmholtz decomposition of a sufficiently smooth tensor field $f$ is the decomposition
\begin{equation}
f
=
g
+
dh
\end{equation}
where $d$ is the symmetrized gradient, 
$g$ is a symmetric $m$-tensor field with $\delta g = 0$ ($\delta$ is the divergence),
and $h$ is a symmetric $(m-1)$-tensor field that vanishes at infinity. Solenoidal injectivity means that the solenoidal part $g$ of $f$ can be recovered from the knowledge of $X_{L}f$.

A method for proving solenoidal injectivity of $X_{L}$ in $\R^n$ is based on the Fourier transform. Given an $m$-tensor field with $X_{L}f=0$ the Fourier slice theorem implies that $\hat f(p)(v,\dots,v) = 0$ when $p \perp v$. A computation then shows that $\hat f(p)(v,\dots,v) = ip \cdot v\hat h(p)(v)$ where $\hat h(p)(\,\cdot\,)$ is a homogeneous polynomial of degree $m-1$. Under sufficient smoothness assumptions this implies that $f = dh$.

Our analysis of the elastic X-ray transform employs similar methods; the Fourier transform, the Fourier slice theorem and Helmholtz decompositions, but instead of the symmetrized gradient $d$ and divergence $\delta$ we use differential operators $H$ and $K$ natural for the elastic transform.

Restricting the elastic ray transform to only the polarization $q = v$ recovers only one operator, $\eXRT{2}_{v,v}$. This ray transform is the usual longitudinal X-ray transform of $4$-tensor field, but for tensor fields only enjoying the elastic symmetries rather than full symmetry. Hence $\eXRT{2}_{v,v}$ alone has two types of obstructions to invertibility; the kernel of the longitudinal transform and the fact that the longitudinal transform only sees the fully symmetric part of the tensor field.

The mixed $(k,l)$ ray transform of an $(k+l)$-tensor field $f$ is defined as
\begin{equation}
\label{eqn:mixed-ray-transform}
X_{M}^{k,l}f(x,v,q)
=
\int_\R
\ip{f(x+tv)}{v^{\otimes k} \otimes q^{\otimes l}}
\,dt
\end{equation}
where $q$ ranges over all vectors perpendicular to $v$. The tensor field $f$ is assumed to be fully symmetric in the first $k$ indices and fully symmetric in the last $l$ indices. We call this $(k,l)$ symmetry type. Tensor fields with $(k,l)$ symmetry type come with a Helmholtz decomposition special to their structure. 
The mixed ray transform is known to be solenoidally injective in the sense of the decomposition. For decompositions of tensor fields, and injectivity and range characterizations of the mixed ray transform see~\cite{dHSUZ2021,dHSZ2019,mishra2025inversion,Sharafutdinov1994,UZ2024a,UZ2024b}.

Restricting the elastic ray transform to only polarizations orthogonal to $v$ recovers the collection of operators $\eXRT{2}_{v,q}$ where $q \in v^\perp$. This collection entails the information of the mixed ray transform, but again for a $4$-tensor fields with elastic symmetries rather than the symmetries natural to the mixed ray transform. Thus this collection of transforms alone comes with two obstructions to uniqueness; the kernel of the mixed ray transform and the fact that the transform only sees the $(k,l)$ symmetric part of the tensor.

Although not directly related to the elastic ray transform, formula~\eqref{eqn:mixed-ray-transform} defines for $k = 0$ 
another interesting class of transforms; the transversal ray transform. The kernel of this ray transform on symmetric tensor fields depends on the dimension. In $2$ dimensions, the transform data reduces to the longitudinal X-ray transform and has non-trivial kernel. In higher, dimensions the kernel is trivial~\cite{Sharafutdinov1994}. One can even take the mixed ray transform as a collection of ray transforms $\{X_{M}^{k,l}\}_{k,l}$. In this case the ray transform has trivial kernel and can be inverted explicitly under certain geometric assumptions~\cite{KMM2019}.

We choose here the approach of formulating the ray transform as a family of operators to conveniently express the transforms as an operator between function spaces. For usage of X-ray transforms as a family of operators rather than a single one see \cite{Ilmavirta2015,IKR2020,KMM2019}.

The choice of the admissible polarization set $\polarizationset{v}{}$ is crucial for the model.
We use $\polarizationset{v}{}=v\R\cup v^\perp$ and we described above (and in more detail in Appendix~\ref{app:spt}) how $v^\perp$ is too large but is the minimal extension that makes the problem fully linear.
If we took $\polarizationset{v}{}=\R^n$, allowing any vector whatsoever to play the role of a polarization, the kernel of $\eXRT{2}$ would consist of only the image of $\eHesse_2$.
In that sense one may say that the operator $\eGrad$ is related polarization limitations.

\subsection{Effects of symmetry}
\label{sec:results-in-different-ranks}

Symmetry assumptions have an effect on the kernel of the inverse problem.
As discussed above, the natural symmetry assumptions for the longitudinal and mixed ray transforms are different from those arising from the physical problem.
Nothing prevents a ray transform from operating on less symmetric tensor fields than usually studied.

If we know that the perturbation enjoys a specific symmetry, the kernel may be different but still determined by the general theorem.
If, for example, the perturbation preserves isotropy (but not homogeneity), then the kernel is trivial as we will see next.
Denote the basic isotropic tensors by $\alpha_{ijkl}=\delta_{ij}\delta_{kl}$ and $\beta_{ijkl}=\delta_{ik}\delta_{jl}+\delta_{ij}\delta_{jk}$.
The two scalar fields (perturbations of Lam\'e parameters)
$\lambda$ and $\mu$ are determined by $\eXRT{2}(\lambda\alpha+\mu\beta)$, because for $q=v$ we have
$\eXRT{2}(\lambda\alpha+\mu\beta)=\eXRT{0}\lambda+2\eXRT{0}\mu$ and for $q\perp v$ we have $\eXRT{2}(\lambda\alpha+\mu\beta)=2\eXRT{0}\mu$.
The operator $\eXRT{0}$ is the familiar invertible Euclidean scalar X-ray transform.

\section{Results in different ranks}
\label{sec:ranks}

\subsection{General rank}
\label{sec:general-rank-proof}

We start with elastic tensor fields of general rank $m \geq 1$. Recall that in Section \ref{sec:introduction} we defined the operators
\begin{equation}
\begin{split}
&
\eHesse_m
\colon
\schwartz(\R^n;\etensor{n}{m-1})
\to
\schwartz(\R^n;\etensor{n}{m})
\\&
\eHesse_m h
=
\elasticsymmetrization(D^2h)
.
\end{split}
\end{equation}
The image of these operators is always in the kernel of the elastic X-ray transform. Here both the operator $\eHesse_m$ and $\eXRT{m}$ act on elastic tensor fields of Schwartz class.

\begin{proposition}
\label{prop:hesse-in-kernel-all-ranks}
Let $m \in \{0,1,2,\ldots\}$. If $\eHesse_m$ is as in \eqref{eq:Op_Hesse} then $\im(\eHesse_m)\subset\ker(\eXRT{m})$.
\end{proposition}

\begin{proof}
Let $h \in \schwartz(\R^n;\etensor{n}{m-1})$. Then for all $x \in \R^n$, $v \in \sphere{n-1}$ and $q \in \polarizationset{v}{}$ we have that
\begin{equation}
\label{eqn:hesse-always-in-kernel-proof}
\begin{split}
\eXRT{m}_{v,q}(\eHesse_mh)(x)
&=
\eXRT{m}_{v,q}(D^2h)(x)
\\
&=
\int_\R
\sum_{\substack{i,j,k_1,l_1,\dots,\\k_{m-1},l_{m-1}}}
\partial_i\partial_j
h_{k_1l_1\cdots k_{m-1}l_{m-1}}(x+tv)
v_iq_j v_{k_1}q_{l_1}
\cdots
v_{k_{m-1}}q_{l_{m-1}}
\,dt
\\
&=
\int_\R
\sum_{\substack{j,k_1,l_1,\dots,\\k_{m-1},l_{m-1}}}
\partial_t[\partial_j
h_{k_1l_1\cdots k_{m-1}l_{m-1}}(x+tv)
q_j v_{k_1}q_{l_1}
\cdots
v_{k_{m-1}}q_{l_{m-1}}]
\,dt
\\
&=
0.
\end{split}
\end{equation}
The last equality in \eqref{eqn:hesse-always-in-kernel-proof} holds since $h \in \schwartz(\R^n;\etensor{n}{m-1})$. Thus $\eXRT{m}(\eHesse_mh) = 0$ for all $h \in \schwartz(\R^n;\etensor{n}{m-1})$ as claimed.
\end{proof}

The natural question is whether this is the whole kernel. We prove that this is the case when $m = 1$. Later in this section we will see that $\eHesse_m$ does not describe the entire kernel for general $m\geq 1$. In particular, we show that the kernel is larger than $\im(\eHesse_2)$ for the rank $m=2$.

\subsection{Rank 0}

Rank $0$ elastic tensor fields are just scalar fields on $\R^n$. For example
\begin{equation}
\schwartz(\R^n;\etensor{n}{0})
=
\schwartz(\R^n).
\end{equation}
In this case the elastic X-ray transform reduces to the usual scalar X-ray transform by definition. Injectivity of this transform is classical (see \cite{Helgason2011,Sharafutdinov1994}).

\subsection{Rank 1}
\label{sec:proof-for-rank-1}

We give two different proofs of ``solenoidal injectivity" using different techniques.
The first proof (Theorem \ref{thm:ims-and-kers-2}) is a cohomological argument using the Poincare lemma and the symmetries of elastic tensor fields.
The second proof (Theorem \ref{thm:ims-and-kers-1}) is based on the Fourier slice theorem, which gives information on the Fourier transform tested against polarization vectors.

For this section, let us denote $\eHesse\coloneqq\eHesse_1$.

\begin{proof}[Proof of Theorem~\ref{thm:ker-characterization-1}]
The inclusion $\im(\eHesse)\subset\ker(\eXRT{1})$ comes from Proposition~\ref{prop:hesse-in-kernel-all-ranks}.
Let us thus take $f\in C_c^\infty(\R^n;\etensor{n}{1})\cap\ker(\eXRT{1})$,
and note that $f$ is fully symmetric.
Our aim is to find a potential $h$ as in the claim of this theorem. 

Suppose first that $n=2$.
The field $f$ is supported in some ball $B$.
Let $x\in\R^2$ and $v\in\sphere{1}$.
For all $h>0$ construct a rectangle $A_h^{x,v}$ so that its boundary consists of a segment of $\gamma_{x,v}$ (the line starting from $x$ with direction $v$), a parallel but reversed copy $\gamma_{x,v}^h=\gamma_{y,-v}$ of this segment shifted by distance $h$, and two line segments of length $h$ connecting these.
The two short segments are chosen to be outside the ball $B$.

For $q\in\polarizationset{v}{}$ denote by $qf$ the one-form with components $(qf)_j=\sum_iq_if_{ij}$.
Changing integration paths outside the support of $f$ and using the fact that $A^{x,v}_h$ collapses to $\gamma_{x,v}$ as $h\to0$, we find via Stokes' theorem that
\begin{equation}
\begin{split}
0
&=
\frac1h
[
\eXRT{1}_{v,q}f(y)
-
\eXRT{1}_{v,q}f(x)
]
\\&=
-\frac1h
\int_{\partial A^{x,v}_h}qf
\\&=
\frac1h
\int_{A^{x,v}_h}\der(qf)
\\&=
\frac1h
\int_{z\in A^{x,v}_h}\star\der(qf)\der z
,
\end{split}
\end{equation}
where $\star$ is the Hodge star that takes 2-forms to 0-forms.
When $h\to0$, the rectangle $A^{x,v}_h$ collapses to the line segment $\gamma_{x,v}$ and we obtain
\begin{equation}
\XRT(\star\der(qf))(\gamma_{x,v})=0,
\end{equation}
where $\XRT$ is the usual planar X-ray transform of scalar functions.

This condition is linear in $q\in\polarizationset{v}{}$ and $\polarizationset{v}{}$ spans the whole space, so we may in fact take $q=e_i$ even though the basis vector $e_i$ might not be in $\polarizationset{v}{}$.
Therefore the one-forms $\omega_i=\sum_{j}f_{ij}\der x^j$ 
satisfy
\begin{equation}
\XRT(\star\der\omega_i)
=
0
.
\end{equation}
Hence, the injectivity of $\XRT$ and $\star$ give that $\der\omega_i=0$.
The Poincar\'e lemma gives $\omega_i=\der w_i$ for some scalar functions $w_i$.

In terms of $f$ this conclusion means that $f_{ij}=\partial_jw_i$.
The symmetry of $f$ gives that $\partial_jw_i=\partial_iw_j$, which means that $w\coloneqq\sum_i w_i\der x^i$ is a closed one-form.
Thus, a second application of the Poincaré lemma gives that $w_i=\partial_ih$ for some scalar function $h\in C_c^\infty(\R^2)$. We conclude that
\begin{equation}
f_{ij}
=
\partial_i\partial_jh = (Hh)_{ij}
\end{equation}
as desired.
In this proof we used a version of the Poincaré lemma for compactly supported differential forms (see~\cite[Corollary 8.3.17]{Conlon2008}). Thus compactness of the support of $h$ is inherited from that of $f$.

Suppose then that $n\geq3$.
The argument above gives that when $f$ is restricted to (the tangent bundle of) any 2-plane $P\subset\R^n$, then there is a compactly supported smooth potential $h_P\colon P\to\C$ so that $f|_P=\eHesse_Ph_P$, where the Hessian is taken along $P$.
We will show that when two 2-planes $P_i$, $i=1,2$, intersect, then $h_{P_1}=h_{P_2}$ on the intersection.
This implies that there is a global function $h\in C_c^\infty(\R^n)$ such that $f=Hh$ and for which $h_P=h|_P$ for each 2-plane $P$.

Consider first the case when $P_1$ and $P_2$ intersect along a line $L$.
The two functions of interest are $h_i\coloneqq h_{P_i}|_L \in C_c^\infty(\R)$.
As both of these two functions are compactly supported, and they have the same second derivative (the quadratic form $f$ evaluated in the direction of $L$), so by the fundamental theorem of calculus they are the same function.

Consider then the case when $P_1\cap P_2=\{x\}$.
Take any lines $L_i\subset P_i$ through $x$ and let $P_3$ be the 2-plane containing the lines $L_1$ and $L_2$.
As $P_1$ and $P_3$ intersect along the line $L_1$, the previous argument shows that $h_{P_1}|_{L_1}=h_{P_3}|_{L_1}$.
Similarly, $h_{P_2}|_{L_2}=h_{P_3}|_{L_2}$.
As $x$ is the intersection point of $L_1$ and $L_2$, we find $h_{P_1}(x)=h_{P_3}(x)=h_{P_2}(x)$ as desired.
\end{proof}

One typically expects that a tensor in the kernel of a ray transform is a first-order derivative of a lower-order field in some manner.
The role of this potential is played by the field $w$ in the proof above.
The symmetry of $f$ and the way that the elastic ray transform is defined forces $w$ to be closed, and a cohomological argument then shows that $w$ is itself a derivative, making $f$ a second-order derivative of a potential $h$.

The formal $L^2$ adjoint of $\eHesse$ is the double divergence $\eHesse^* \colon \schwartz(\R^n;\etensor{n}{1}) \to \schwartz(\R^n)$ defined by
\begin{equation}
\eHesse^*f
=
\sum_{i,j}
\partial_i\partial_j f_{ij}.
\end{equation}
In the following theorem the operators $\eHesse$ and $\eXRT{1}$ act on elastic $1$-tensor fields of Schwartz class.

\begin{proposition}
\label{thm:ims-and-kers-1}
For all $n\geq1$ the following hold:
\begin{enumerate}
\item\label{claim:ker1-vector}
$\im(\eHesse)\subset\ker(\eXRT{1})$.
\item\label{claim:ker2-vector}
$\ker(\eHesse^*)\cap\ker(\eXRT{1})=0$.
\end{enumerate}
\end{proposition}

\begin{proof}
Claim~\ref{claim:ker1-vector}: The claim is a special case of Proposition~\ref{prop:hesse-in-kernel-all-ranks} ($m=1$).

Claim~\ref{claim:ker2-vector}: Since $\eXRT{1}f = 0$, given $p \ne 0$, it follows from the Fourier slice theorem that
\begin{equation}
\sum_{i,j}
\hat f_{ij}(p)v_iq_j
=
0
\end{equation}
for all $v \in p^\perp$ and $q \in \polarizationset{v}{}$. This can be written as $q^T\hat f(p) v = 0$ for all $v \in p^\perp$ where we have fixed $p$ and interpreted the Fourier transform of $f$ evaluated at $p$ as a matrix. Since $\polarizationset{v}{}$ contains a basis for $\R^n$ we get $\hat f(p)v = 0$ for all $v \in p^\perp$. Therefore $\hat f(p) = \hat h(p)pp^T$, where $\hat h(p) = \abs{p}^{-4}p^T\hat f(p)p \in \R$.
Thus since $\eHesse^*f = 0$ we get that
\begin{equation}
0
=
-
\sum_{i,j}
p_ip_j\hat f_{ij}(p)
=
-\hat h(p)
\abs{p}^4.
\end{equation}
Then we have shown that $\hat{f}_{ij}(p) = \hat h(p)p_ip_j = 0$ for all $i$ and $j$, which proves that $f = 0$ since $f \in \schwartz(\R^n;\etensor{n}{1})$.
\end{proof}

\begin{remark}
We could pair Theorem \ref{thm:ims-and-kers-1} with a suitable Helmholtz type decomposition to give a second proof that $\ker(\eXRT{1}) = \im(\eHesse)$. However, we do not pursue this in the rank $1$ case.
\end{remark}

\subsection{Rank 2}
\label{sec:proof-for-rank-2}

In this section, we prove our main result (Theorem \ref{thm:ims-and-kers-2}) on the elastic X-ray transform of elastic tensor fields of rank $2$.
Recall that in Section \ref{sec:introduction} we defined the operators
\begin{equation}
\begin{split}
&
\eHesse_2
\colon
\schwartz(\R^n;\etensor{n}{1})
\to
\schwartz(\R^n;\etensor{n}{2})
\\&
\eHesse_2 h
=
\elasticsymmetrization(D^2h)
\end{split}
\end{equation}
and
\begin{equation}
\begin{split}
&
\eGrad
\colon
\schwartz(\R^n;\R^n)
\to
\schwartz(\R^n;\etensor{n}{2})
\\&
\eGrad W
=
\elasticsymmetrization(DW\otimes\idmatrix)
.
\end{split}
\end{equation}
In this section we denote $\eHesse \coloneqq \eHesse_2$. The formal $L^2$ adjoints of the operators are
\begin{equation}
\begin{split}
&
\eHesse^* f
\colon
\schwartz(\R^n;\etensor{n}{2})
\to
\schwartz(\R^n;\etensor{n}{1})
\\&
(\eHesse^*f)_{ij}
=
\sum_{k,l}\partial_k\partial_l f_{ijkl}
\end{split}
\end{equation}
and
\begin{equation}
\begin{split}
&
\eGrad^*
\colon
\schwartz(\R^n;\etensor{n}{2})
\to
\schwartz(\R^n;\R^n)
\\&
(\eGrad^* f)_i
=
-\sum_{j,k} \partial_jf_{ijkk}
.
\end{split}
\end{equation}
Thus, for all $f \in \schwartz(\R^n;\etensor{n}{2})$, $h \in \schwartz(\R^n;\etensor{n}{1})$ and $W \in \schwartz(\R^n;\R^n)$ we have that
\[
\iip{\eHesse h}{f}_{L^2} = \iip{h}{\eHesse^* f}_{L^2} 
\text{ and } 
\iip{\eGrad W}{f}_{L^2} = \iip{h}{\eGrad^* f}_{L^2}.
\]

Before proving Theorem~\ref{thm:ims-and-kers-2} we present three lemmas.
Lemma \ref{lma:fst} is an adaptation of the Fourier slice theorem to the elastic X-ray transform.
Lemma \ref{lma:no-pw-kernel} shows that in the conclusion of the Fourier slice theorem a non-trivial tensor field cannot vanish pointwise. Lemma \ref{lma:reduction} manipulates some rather technical conditions encountered in the proof of Theorem \ref{thm:ims-and-kers-2} into a more manageable form.

\begin{lemma}[Fourier slice theorem]
\label{lma:fst}
The following are equivalent for $f\in\schwartz(\R^n;\etensor{n}{2})$:
\begin{enumerate}
\item $\eXRT{2}f=0$.
\item $\hat f(0)=0$ and for all $p\neq0$ we have that
\begin{equation}
\sum_{i,j,k,l}\hat f_{ijkl}(p)v_iq_jv_kq_l
=
0
\end{equation}
whenever $v\perp p$ and $q\in\polarizationset{v}{}$.
\end{enumerate}
\end{lemma}

The proof of Lemma~\ref{lma:fst} is  immediate upon taking the Fourier transform and we omit the details in here.

\begin{lemma}[No pointwise kernel; proven in Section~\ref{sec:cas-proofs}]
\label{lma:no-pw-kernel}
Let $n\geq1$.
Let $a\in\etensor{n}{2}$ and suppose that for all $v\in\sphere{n-1}$ we have that
\begin{equation}
\label{eq:lma:no-pw-kernel}
\sum_{i,j,k,l}
a_{ijkl}v_iq_jv_kq_l=0,
\end{equation}
whenever $q\in\polarizationset{v}{}$.
Then $a=0$.
\end{lemma}

\begin{lemma}[Proven in Section \ref{sec:cas-proofs}]
\label{lma:reduction}
Let $n \geq 1$ and let $a \in \etensor{n}{2}$. Let $p = (1,0,\dots,0)\in\R^n$. Suppose that $a_{11ij}(p) = 0$ for all $i,j$. Then the following systems of linear equations are equivalent:
\begin{enumerate}
    \item For all $v\in\sphere{n-1}\cap p^\perp$ we have that
    \begin{equation}
    \label{eq:vv1}
    \sum_{i,j,k,l}a_{ijkl}(p)v_iv_jv_kv_l=0 
    \end{equation}
    and
    \begin{equation}
    \label{eq:vv2}
    \sum_{i,j,k,l}
    a_{ijkl}(p)v_iq_jv_kq_l=0
    \quad\text{for all }q\perp v
    .
    \end{equation}

    \item For all indices $i,j,k,l\geq2$ we have that
    \begin{equation}
    \label{eqn:reduction1}
    \begin{cases}
    a_{ijkl}(p)=0
    \\
    a_{1i1j}(p)=0
    \\
    a_{1ijk}(p)=0 \text{ when }j\neq k
    \\
    a_{1i22}(p)=a_{1i33}(p)=\dots=a_{1inn}(p)
    .
    \end{cases}
    \end{equation}
\end{enumerate}
\end{lemma}

\begin{proof}[Proof of Theorem~\ref{thm:ims-and-kers-2}]
Claim~\ref{claim:ker1}: The inclusion $\im(\eHesse) \subset \ker(\eXRT{2})$ is a special case of Proposition~\ref{prop:hesse-in-kernel-all-ranks} ($m = 2$). To prove the inclusion $\im(\eGrad) \subset \ker(\eXRT{2})$ let $W \in \schwartz(\R^n;\R^n)$. Then
\begin{equation}
\label{eqn:k-in-kernel}
\begin{split}
\eXRT{2}_{v,q}(\eGrad W)(x)
=
\int_\R
\sum_{i,j,k,l}
\frac12
[
\partial_iW_{j}(x+tv)
+\partial_jW_{i}(x+tv)
]
v_iq_j\ip{v}{q}\der t
.
\end{split}
\end{equation}
Recall that we have either $\ip{v}{q}=0$ or $q=v$. Hence the integral above either vanishes or the integrand is of the form
\begin{equation}
\partial_t[W_{i}(x+tv)v_i].
\end{equation}
Therefore, also in the latter case the integral vanishes due to the fundamental theorem of calculus and since $W$ is in the Schwartz class.

Claim~\ref{claim:ker2}:
Suppose $f\in\schwartz(\R^n;\etensor{n}{2})$ is in $\ker(\eHesse^*)\cap\ker(\eGrad^*)\cap\ker(\eXRT{2})$.
We aim to show that $\hat f(p)=0$ for all $p\in\R^n$.
For $p=0$ this is contained in lemma~\ref{lma:fst}, so we take $p\neq0$.
The kernels of the operators involved are invariant under scaling and rotation, so it suffices to prove the statement for $p=(1,0,\dots,0)$.

We have
\begin{equation}
\label{eq:vv3}
\begin{split}
0
&=
\widehat{\eHesse^*f}_{ij}(p)
\\&=
-\sum_{k,l}\hat f_{ijkl}(p)p_kp_l
\\&=
-\hat f_{ij11}(p)
\end{split}
\end{equation}
and
\begin{equation}
\label{eq:vv4}
\begin{split}
0
&=
\widehat{\eGrad^*f}_{i}(p)
\\&=
i\sum_{j,k}\hat f_{ijkk}(p)p_j
\\&=
i\sum_{k}\hat f_{i1kk}(p)
.
\end{split}
\end{equation}
On the other hand, for all $v\in\sphere{n-1}\cap p^\perp$ lemma~\ref{lma:fst} gives
\begin{equation}
\label{eq:vv5}
\sum_{i,j,k,l}\hat f_{ijkl}(p)v_iv_jv_kv_l=0
\end{equation}
and
\begin{equation}
\label{eq:vv6}
\sum_{i,j,k,l}
\hat f_{ijkl}(p)v_iq_jv_kq_l=0
\quad\text{for all }q\perp v
.
\end{equation}
Conditions~\eqref{eq:vv3}, ~\eqref{eq:vv5}, and~\eqref{eq:vv6} imply the equations \eqref{eq:vv1} and \eqref{eq:vv2} of Lemma~\ref{lma:reduction} when $a=\hat f$. Therefore, we also have that
\begin{equation}
\begin{cases}
\hat f_{ijkl}(p)=0
\\
\hat f_{1i1j}(p)=0
\\
\hat f_{1ijk}(p)=0 \text{ when }j\neq k
\\
\hat f_{1i22}(p)=\hat f_{1i33}(p)=\dots=\hat f_{1inn}(p)
\end{cases}
\end{equation}
for all indices $i,j,k,l\geq2$. Combining these with the conditions obtained in~\eqref{eq:vv3} and~\eqref{eq:vv4} wee see that $\hat f_{ijkl}(p)$ indeed vanish for all possible values of $i,j,k,l\in \{1,2,\ldots,n\}$.
\end{proof}

\section{The Helmholtz decomposition of rank 2 elastic tensor fields}
\label{sec:helmholtz}

In the previous section we proved that the elastic X-ray transform $\eXRT{2}$ is injective on solenoidal fields, i.e. fields $S \in \schwartz(\R^n;\etensor{n}{2})$ with $\eHesse_2^*S = \eGrad^*S = 0$, and that potentials, i.e. fields of the form $\eHesse_2 h + \eGrad W$ where $W \in \schwartz(\R^n;\R^n)$ and $h \in \schwartz(\R^n;\etensor{n}{1})$, are invisible under the elastic X-ray transform $\eXRT{2}$ (see Theorem \ref{thm:ims-and-kers-2}). This is not yet enough to conclude that the kernel of the transform is solely comprised of potential fields. We remedy this in the current section by proving that an elastic $2$-tensor field $f$ enjoys a unique decomposition into a solenoidal and a potential part in the former sense.
For technical reasons we prove the existence of such a decomposition under the assumption that $n \geq 3$. This assumption is only required to prove Lemma~\ref{lma:schwartz-decomposition}.

Our proof of existence and uniqueness of such decompositions is inspired by that of \cite[Chapter 2.3.]{Sharafutdinov1994} but the technical details are quite different.
The main difference and difficulty in our case is that our potential fields have two possibly overlapping parts related to the two differential operators $\eHesse_2$ and $\eGrad$. The proof proceeds in three steps. First, we consider decompositions of elastic tensors, which are then in the second step used to decompose the Fourier transforms of elastic tensor fields pointwise. In the last step we promote the decomposition into the space of $L^2$ fields via density arguments.

This section only concerns elastic $2$-tensor fields so we continue our convention of omitting subindex from the differential operator $\eHesse \coloneqq \eHesse_2$.

\subsection{Pointwise decomposition on the Fourier side}
\label{sec:pointwise-decompositions}

We are looking for decompositions of elastic $2$-tensor fields as $f = Hh + KW + S$ with $K^*S = H^*S = 0$, where $W$ is a vector field, $h$ is an elastic $1$-tensor field and $f$ and $S$ are elastic $2$-tensor fields. If such a decomposition exists for a sufficiently regular field, then taking Fourier transform gives $\hat f(p) = \hsymb{p}\hat h(p) + \ksymb{p}\widehat W(p) + \widehat S(p)$ with $\ktsymb{p}\widehat{S}(p) = \htsymb{p}\widehat{S}(p) = 0$. The linear maps $\ksymb{p}$ and $\hsymb{p}$ and their duals are defined below. We begin by proving that such decompositions exist pointwise on the Fourier domain.

Let a non-zero vector $p \in \R^n$ be given. We define the linear maps $\ksymb{p} \colon \R^n \to \etensor{n}{2}$ and $\hsymb{p} \colon \etensor{n}{1} \to \etensor{n}{2}$ by
\begin{equation}
(\ksymb{p}W)_{ijkl}
=
\frac14(p_iW_j+p_jW_i)\delta_{kl}
+
\frac14(p_iW_j+p_jW_i)\delta_{ij}
\end{equation}
and
\begin{equation}
(\hsymb{p}h)_{ijkl}
=
\frac12(p_ip_jh_{kl}+p_kp_lh_{ij}).
\end{equation}
Thus, for any $W \in \schwartz(\R^n;\R^n)$ and $h \in \schwartz(\R^n;\etensor{n}{1})$ the equations \eqref{eqn:oper-k}  and \eqref{eq:Op_Hesse} imply $\widehat{KW}(p) = i\ksymb{p}\widehat{W}(p)$ and $\widehat{Hh}(p) = -\hsymb{p}\widehat{h}(p)$ respectively. We define linear maps $\ktsymb{p} \colon \etensor{n}{2} \to \R^n$ and $\htsymb{p} \colon \etensor{n}{2} \to \etensor{n}{1}$ by
\begin{equation}
(\ktsymb{p}T)_i
=
\sum_{\beta,\lambda,\mu}
p_\beta\delta_{\lambda\mu}T_{i\beta\lambda\mu}
\end{equation}
and
\begin{equation}
(
\htsymb{p}T
)_{ij}
=
\sum_{ij\lambda\mu}
T_{ijkl}p_\lambda p_\mu.
\end{equation}
These are the algebraic duals of $\ksymb{p}$ and $\hsymb{p}$ as one easily verifies that
\begin{equation}
\ip{\ksymb{p}W}{T}_{\etensor{n}{2}}
=
\ip{W}{\ktsymb{p}T}_{\R^n}
\quad
\text{and}
\quad
\ip{\hsymb{p}h}{T}_{\etensor{n}{2}}
=
\ip{h}{\htsymb{p}T}_{\etensor{n}{1}}
\end{equation}
where we have the natural inner products of $\R^n$, $\etensor{n}{1}$ and $\etensor{n}{2}$.

\begin{lemma}
\label{lma:pw-decomposition}
Fix a non-zero vector $p \in \R^n$. Then the space of elastic $2$-tensors has the (interior) direct sum decomposition
\begin{equation}
\etensor{n}{2}
=
\kuvak{p}
\oplus
\kuvai{p}
\oplus
\ydinki{p}
\end{equation}
where
\begin{equation}
\begin{split}
\kuvak{p} &= \{\ksymb{p}W \,:\, W \in \R^n, W \perp p\},
\\
\kuvai{p} &= \{\hsymb{p}h \,:\, h \in \etensor{n}{1}\},
\\
\ydinki{p} &= \{S \in \etensor{n}{2}\,:\, \ktsymb{p}S = \htsymb{p}S = 0\}
.
\end{split}
\end{equation}
In addition, the projections onto the subspaces $\kuvak{p}$, $\kuvai{p}$ and $\ydinki{p}$ have $0$-homogeneous dependence on $p$.
\end{lemma}

\begin{proof}
Consider the linear map
$
\phi
\colon
\R^n \times \etensor{n}{1}
\to
\etensor{n}{2}
$
defined by $\phi(W,h) = \ksymb{p}W + \hsymb{p}h$ for all $W \in \R^n$ and $h \in \etensor{n}{1}$.
We equip $\R^n \times \etensor{n}{1}$ with the product inner product. Then the algebraic dual of $\phi$ is the linear map
$
\phi^*
\colon
\etensor{n}{2}
\to
\R^n \times \etensor{n}{1}
$
defined by $\phi^*(T) = (\ktsymb{p}T,\htsymb{p}T)$.
Therefore, since the vector spaces are finite dimensional, we get that
\begin{equation}
\begin{split}
\etensor{n}{2}
&=
\im(\phi) \oplus \im(\phi)^\perp
\\
&=
\im(\phi) \oplus \ker(\phi^*)
\\
&=
(\im(\hsymb{p}) + \im(\ksymb{p}))
\oplus
(\ker(\htsymb{p})\cap\ker(\ktsymb{p})).
\end{split}
\end{equation}
We have shown that $\etensor{n}{2} = D_p \oplus \ydinki{p}$ where $D_p = \im(\hsymb{p}) + \im(\ksymb{p})$. Next, we prove that $D_p = \kuvak{p} \oplus \kuvai{p}$.

Given $T \in D_p$ we know that there are $W \in \R^n$ and $h \in \etensor{n}{1}$ so that $T = \ksymb{p}W + \hsymb{p}h$.
The vector $W$ can be uniquely decomposed as $W_i = p_ig + S_i$ where $g \in \R$ and $S \in \R^n$ is orthogonal to $p$.
Then we can verify that the decomposition
\begin{equation}
T = \ksymb{p}S + (\ksymb{p}(pg) + \hsymb{p}h)
\end{equation}
is unique and $\ksymb{p}S \in \kuvak{p}$ and $\ksymb{p}(pg) + \hsymb{p}h = \hsymb{p}(g\delta + h) \in \kuvai{p}$ where $(g\delta)_{ij} = g\delta_{ij}$ proving that $D_p = \kuvak{p} \oplus \kuvai{p}$.

Lastly, the fact that the projections onto the subspaces are $0$-homogeneous in $p$ follows from the fact that the definitions for the subspaces are rotation invariant. In more detail, the equations defining $\ydinki{p}$ are rotation invariant, showing that the projection on $\ydinki{p}$ is $0$-homogeneous. Thus also projection onto $D_p$ is. Then it follows that projections onto $\kuvak{p}$ and $\kuvai{p}$ are $0$-homogeneous since their sum is direct.
\end{proof}

The projections onto the subspaces $\kuvak{p}$, $\kuvai{p}$ and $\ydinki{p}$ related to the direct sum decomposition in Lemma \ref{lma:pw-decomposition} will be denoted by $\pi^{\kuvak{p}} \colon \etensor{n}{2} \to \kuvak{p}$, $\pi^{\kuvai{p}} \colon \etensor{n}{2} \to \kuvai{p}$ and $\pi^{\ydinki{p}} \colon \etensor{n}{2} \to \ydinki{p}$ respectively.

Before moving onto decompositions of fields we prove that the linear maps $\ksymb{p}$ and $\hsymb{p}$ are left-invertible. This will help in analyzing the smoothness properties of the components when fields are decomposed the following section.
We define  $\kinverse{p} \colon \etensor{n}{2} \to \R^n$ by
\begin{equation}
\label{eqn:k-inverse}
(\kinverse{p} T)_i
=
4\abs{p}^{-4}
\sum_{\beta,\lambda,\mu}
T_{i\beta\lambda\mu}p_\beta p_\lambda p_\mu
\end{equation}
and $\iinverse{p} \colon \etensor{n}{2} \to \etensor{n}{1}$ by
\begin{equation}
\label{eqn:h-inverse}
(\iinverse{p}T)_{ij}
=
2\abs{p}^{-4}
\sum_{\lambda,\mu}T_{ij\lambda\mu}p_\lambda p_\mu
-
\abs{p}^{-8}
p_ip_j
\sum_{\alpha,\beta,\lambda,\mu}T_{\alpha\beta\lambda\mu}p_\alpha p_\beta p_\lambda p_\mu.
\end{equation}

\begin{lemma}
\label{lma:inverses}
Fix a non-zero vector $p \in \R^n$. Then we have $\kinverse{p}(\ksymb{p}W) = W$ and $\iinverse{p}(\hsymb{p}h) = h$ for all $W \in \R^n$ with $W \perp p$ and $h \in \etensor{n}{1}$.
\end{lemma}

\begin{proof}
A direct computation verifies that $\kinverse{p} \circ \ksymb{p} = id_{p^\perp}$ and $\iinverse{p} \circ \hsymb{p} = id_{\etensor{n}{1}}$ where $p^\perp$ is the subspace orthogonal to $p$ in $\R^n$.
\end{proof}

In particular, Lemma \ref{lma:inverses} implies that if the decomposition of $T \in \etensor{n}{2}$ as in Lemma \ref{lma:pw-decomposition} is $T = \ksymb{p}W + \hsymb{p}h + S$ then $W = \kinverse{p}(\pi^{\kuvak{p}}T)$ and $h = \iinverse{p}(\pi^{\kuvai{p}}T)$. Here the $p$-dependence of $\kinverse{p} \circ \pi^{\kuvak{p}}$ is homogeneous of degree $-1$ and the $p$-dependence of $\iinverse{p} \circ \pi^{\kuvai{p}}$ is of degree $-2$.

\subsection{Global decomposition on the Fourier side}
\label{sec:global-decompositions-Fourier}

Having decompositions of elastic $2$-tensors proven in the previous section we are ready to show the existence of decompositions of elastic $2$-tensor fields in a certain function space $\avaruus_2$ defined below. All Schwartz tensor fields are examples of elements of $\avaruus_2$. The method of proof (see Lemma \ref{lma:schwartz-decomposition}) is to analyze the point-dependence using homogeneity properties of components in the decomposition provided by Lemma \ref{lma:pw-decomposition}. This happens on the Fourier domain.

We define $\avaruus_2$ to be the space of elastic $2$-tensor fields $f$ satisfing the properties
\begin{enumerate}
    \item $f \in \smooth{2} \cap \ltwo{2}$,
    \item $\abs{f(x)} \leq C(1 + \abs{x})^{1-n}$ for all $x \in \R^n$, and
    \item $\hat f(p)$ is smooth away from $p = 0$ and rapidly decreasing.
\end{enumerate}
In particular, $\schwartz(\R^n;\etensor{n}{2})\subseteq \avaruus_2$ so $\avaruus_2$ is dense in $\ltwo{2}$. In addition, we define the space $\avaruus^H_1$ to be the space of smooth elastic $1$-tensor fields $h$ so that $\eHesse h \in \avaruus_2$, and the space $\avaruus^K_0$ to be the space of smooth vector fields $W$ so that $\eGrad W \in \avaruus_2$.

\begin{lemma}
\label{lma:schwartz-decomposition}
Let $n \geq 3$.
Any elastic $2$-tensor field $f \in \avaruus_2$ can be uniquely decomposed as $f = P + S$ with $H^*S = K^*S = 0$ where $P,S \in \avaruus_2$ so that $P = KW + Hh$ for some $W \in \avaruus^K_0$ and $h \in \avaruus^H_1$.
In addition, all first partial derivatives of $S$, $W$ and $h$ are in $L^\infty$.
\end{lemma}

\begin{proof}
For any non-zero $p \in \R^n$, it follows from Lemma \ref{lma:pw-decomposition} that, the Fourier transform of $f$ can be uniquely decomposed as
\begin{equation}
\hat f(p)
=
\ksymb{p}\widehat W(p)
+
\hsymb{p}\hat h(p)
+
\widehat S(p)
\end{equation}
where $\widehat W(p) \perp p$ and $\hat h(p) \in \etensor{n}{1}$ and $\widehat S(p) \in \etensor{n}{2}$ is so that $\ktsymb{p}\widehat S(p) = \htsymb{p}\widehat S(p) = 0$. Moreover, by Lemma \ref{lma:inverses} we have that
\begin{equation}
\label{eqn:proof-of-sch-dec-1}
\widehat W(p)
=
(\kinverse{p} \circ \pi^{\kuvak{p}})\hat{f}(p),
\quad
\hat h(p)
=
(\iinverse{p} \circ \pi^{\kuvai{p}})\hat{f}(p)
\quad
\text{and}
\quad
\widehat S(p) = \pi^{\ydinki{p}} \hat f(p).
\end{equation}
Equation~\eqref{eqn:proof-of-sch-dec-1} defines integrable functions since $n \geq 3$. This is true since $\hat f$ is rapidly decreasing, while near the origin $\hat S$ is bounded, $\abs{\widehat W(p)} \leq C\abs{p}^{-1}$, and $\abs{\hat h(p)} \leq C\abs{p}^{-2}$ due to equations~\eqref{eqn:k-inverse} and~\eqref{eqn:h-inverse}.

Since the Fourier transform of $f$ is rapidly decreasing and smooth when $p \ne 0$, we get from \eqref{eqn:proof-of-sch-dec-1} and homogeneity that $\widehat W(p)$, $\hat h(p)$ and $\widehat S(p)$ are smooth when $p \ne 0$ and rapidly decreasing. 
Thus taking the inverse Fourier transform the formulas
\begin{equation*}
W(x)
\coloneqq
(-i)
\mathcal{F}^{-1}(\widehat W)(x),
\quad
h(x)
\coloneqq
(-1)
\mathcal{F}^{-1}(\hat h)(x)
\quad
\text{and}
\quad
S(x)
\coloneqq
\mathcal{F}^{-1}(\widehat{S})(x)
\end{equation*}
define $W \in C^\infty(\R^n;\R^n)$, $h \in \smooth{1}$ and $S \in \smooth{2}$ so that
\begin{equation}
f(x)
=
KW(x)
+
Hh(x)
+
S(x)
\quad
\text{with}
\quad
K^*S(x) = H^*S(x) = 0.
\end{equation}
It remains to prove that $P,S \in \avaruus_2$, $W \in \avaruus^K_0$ and $h \in \avaruus^H_1$. To accomplish this it suffices to prove that $KW,Hh,S \in \avaruus_2$.

For any multi-index $\delta$ it holds that
\begin{equation}
\partial_p^\delta\widehat{S}(p)
=
\sum_{\alpha,\beta,\lambda,\mu}
\left(
\sum_{\epsilon \leq \delta}
\binom{\delta}{\epsilon}
\partial_p^\epsilon
(\pi^C_p)^{\alpha\beta\lambda\mu}_{ijkl}
\partial^{\delta - \epsilon}_p\hat{f}_{\alpha\beta\lambda\mu}(p)
\right).
\end{equation}
Thus, due to homogeneity of $\pi^{\ydinki{p}}$, we obtain the estimate
\begin{equation}
\abs{\partial^\delta_p\widehat{S}(p)}
\leq
C\abs{p}^{-\abs{\delta}}.
\end{equation}
This together with rapid decreasing of $\hat{f}$ proves that $\widehat{S} \in \ltwo{2}$. By a similar computation, and using $0$-homogeneity of the other two projections, we prove that $\ksymb{p}\widehat{W},\hsymb{p}\hat{h}\in \ltwo{2}$. Thus by the Plancherel formula $S,KW,Hh \in \ltwo{2}$. Furthermore, these estimates for $\widehat{S}$, $\ksymb{p}\widehat{W}$ and $\hsymb{p}\hat{h}$ prove that $S$, $KW$ and $Hh$ satisfy the estimates
\begin{equation}
\abs{S(x)},\abs{KW(x)},\abs{Hh(x)}
\leq
C(1+\abs{x})^{1-n}.
\end{equation}
Since smoothness and decay properties of the Fourier transforms were already established, we conclude that $S,KW,Hh \in \avaruus_2$.

As the last step we prove that the derivatives of $S$, $W$ and $h$ are bounded. We will only prove this for $h$. The same estimates hold for $S$ and $W$. Taking the Fourier transform and the inverse Fourier transform yields
\begin{equation}
\abs{\partial_ih(x)}
\leq
\int_{\R^n}
\abs{p}\abs{\hat h(p)}
\,dp.
\end{equation}
Since $\hat h$ is rapidly decreasing, the right-hand side above is finite if and only if
\begin{equation}
\int_{B(0,1)}
\abs{p}\abs{\hat{h}(p)}
\,dp
<
\infty.
\end{equation}
We know from Lemma \ref{lma:inverses} that $\hat{h}(p) = (\iinverse{p}\pi^B_p)\hat{f}(p)$ so homogeneity implies that $\abs{\hat{h}(p)} \leq C\abs{p}^{-2}$. Thus
\begin{equation}
\int_{B(0,1)}
\abs{p}\abs{\hat{h}(p)}
\,dp
\leq
C
\int_{B(0,1)}
\abs{p}^{-1}
\,dp
\end{equation}
where the right-hand side is finite since $n \geq 3$.
Thus the derivative is bounded as claimed.
\end{proof}

\subsection{Decomposition of $L^2$-tensor fields}
\label{sec:helmholt-final-proof}

We let $n \geq 3$ throughout this section.
The main content of the section is an orthogonal direct sum decomposition of the space of $L^2$ elastic tensor fields into solenoidal and potential parts. The proof of existence of such decomposition is mainly based on density arguments and Lemma \ref{lma:schwartz-decomposition}. We end the section with the proof of Theorem \ref{thm:ker-characterization-2}.

Let $\sol_2$, the subspace of solenoidal fields in $\ltwo{2}$, be the $L^2$-closure of the set of elastic $2$-tensor fields $S \in \avaruus_2$ so that $K^*S = H^*S = 0$. Let $\pot_2$, the subspace of potential fields in $\ltwo{2}$, be the $L^2$-closure of the set $H\avaruus^H_1 + K\avaruus^K_0$.

We need the following auxiliary lemma to prove our decomposition. Here $\nu$ is denotes the inward unit normal to the boundary of a smooth domain $D \subset \R^n$.

\begin{lemma}
\label{lma:greens}
Let $D$ be an open and bounded subset of $\R^n$ with smooth boundary. 
Then for all $S \in \smooth{2}$, $h \in \smooth{1}$, and $W \in \smooth{0}$ we have that
\begin{equation}
\int_{\overline{D}}
\left[
\ip{S}{KW}
-
\ip{K^*S}{W}
\right]
\,dx
=
\sum_{i,j,k}
\int_{\partial\overline{D}}
\nu_jS_{ijkk}W_i
\,dS
\end{equation}
and
\begin{equation}
\int_{\overline{D}}
\left[
\ip{S}{Hh}
-
\ip{H^*S}{h}
\right]
\,dx
=
\sum_{i,j,k,l}
\int_{\partial\overline{D}}
\left(
\nu_iS_{ijkl}\partial_jh_{kl}
-
\nu_j\partial_iS_{ijkl}h_{kl}
\right)
\,dS.
\end{equation}
\end{lemma}

\begin{proof}
The proof is a straightforward application of the Green's identity after computing the differences of the pointwise inner products in coordinates.
\end{proof}

\begin{proposition}
\label{thm:helmholtz}
Let $n \geq 3$.
The space of $L^2$-regular elastic $2$-tensor fields enjoys the orthogonal direct sum decomposition
\begin{equation}
\ltwo{2} = \pot_2 \oplus \sol_2
\end{equation}
where the subspaces $\pot_2$ and $\sol_2$ are defined above.
\end{proposition}

\begin{proof}
First, we prove that the subspaces $\sol_2$ and $\pot_2$ are orthogonal to each other. Let $P \in \pot_2$ and $S \in \sol_2$. Then there are $S_m \in \avaruus_2$ so that $K^*S_m = H^*S_m = 0$ so that $S_m \to S$ in $\ltwo{2}$. Also, there are $P_m = Hh_m + KW_m$ where $h_m \in \avaruus^H_1$ and $W_m \in \avaruus^K_0$ so that $P_m \to P$ in $\ltwo{2}$. It follows from Lemma \ref{lma:greens} that
\begin{equation}
\label{eqn:greens-applied}
\begin{split}
\int_{\abs{x} \leq R}
\ip{S_m}{P_m}
\,dx
&=
\sum_{i,j,k}
\int_{\abs{x} = R}
\nu_j(S_m)_{ijkk}(W_m)_i
\,dS
\\
&\quad + 
\sum_{i,j,k,l}
\int_{\abs{x} = R}
\left(
\nu_i(S_m)_{ijkl}\partial_j(h_m)_{kl}
-
\nu_j\partial_i(S_m)_{ijkl}(h_m)_{kl}
\right)
\,dS.
\end{split}
\end{equation}
Due to Lemma \ref{lma:schwartz-decomposition}, the derivatives of $S$ and $h$ are bounded and $S$, $W$ and $h$ vanishes at infinity. Thus, the right-hand side of \eqref{eqn:greens-applied} goes to zero as $R \to \infty$. Thus, $\iip{S_m}{P_m}_{\ltwo{2}} = 0$  and after passing to the limit $m \to \infty$ we get $\iip{S}{P}_{\ltwo{2}} = 0$.

The fact that $\pot_2 \oplus \sol_2 = \ltwo{2}$ follows from the density of $\avaruus_2$ in $\ltwo{2}$ and from Lemma \ref{lma:schwartz-decomposition}: Let $f \in \ltwo{2}$ and choose a sequence $f_k \in \avaruus_2$ such that $f_k \to f$ in $\ltwo{2}$. Then, by Lemma \ref{lma:schwartz-decomposition}, there are $P_ k \in \avaruus_2$ and $S_k \in \avaruus_2$ so that
\begin{equation}
\label{eqn:l2-decomposition-proof}
f_ k = P_k + S_k,
\quad
\text{with}
\quad
H^*S_k=K^*S_k=0
\end{equation}
and $P_k = KW_k + Hh_k$ for some $W_k \in \avaruus^K_0$ and $h_k \in \avaruus^H_1$.

We have
\begin{equation}
f_l - f_k
=
(P_k - P_l)
+
(S_k - S_k).
\end{equation}
Due to the orthogonality proved earlier we see that
\begin{equation}
\norm{f_k-f_l}^2
=
\norm{P_k - P_l}^2
+
\norm{S_k - S_k}^2.
\end{equation}
Thus since $(f_k)$ is Cauchy, also the sequences $(P_k)$ and $(S_k)$ are Cauchy, and consequently converge to some $P \in \ltwo{2}$ and $S \in \ltwo{2}$. In fact, $P \in \pot_2$ and $S \in \sol_2$ respectively. Then passing to the limit $k \to \infty$ in \eqref{eqn:l2-decomposition-proof} proves the existence of the desired sum decomposition.
\end{proof}

Before proving the main theorem of this section (Theorem~\ref{thm:ker-characterization-2}) we note that all tensor fields in $\pot_2$ can be written as $KW + Hh$ where $W \in H^1_K$ and $h \in H^2_H$. This follows by continuity of $K$ and $H$. The space $H^1_K$ is the completion of $\schwartz(\R^n;\R^n)$ with respect to the norm
\begin{equation}
\norm{W}^2
\coloneqq
\norm{W}^2_{L^2(\R^n;\R^n)}
+
\norm{\eGrad W}^2_{\ltwo{2}}
\end{equation}
and the space $H^2_H$ is the completion of the space $\schwartz(\R^n;\etensor{n}{1})$ with respect to the norm
\begin{equation}
\norm{h}^2
\coloneqq
\norm{h}^2_{\ltwo{1}}
+
\norm{\eHesse h}^2_{\ltwo{2}}.
\end{equation}

\begin{proof}[Proof of Theorem~\ref{thm:ker-characterization-2}]
Since $X_{v,q}$ is continuous from $\ltwo{2}$ to $L^2(v^\perp;\R)$ for all $v \in \sphere{n-1}$ and $v \in \polarizationset{v}{}$, and $\schwartz(\R^n;\etensor{n}{2})$ is dense in $\ltwo{2}$, the following claims on $\ltwo{2}$ are consequences of Theorem~\ref{thm:ims-and-kers-2}: 
\begin{equation}
\label{eqn:l2-pot-invisible}
\im(\eHesse)
+
\im(\eGrad)
\subseteq
\ker(\eXRT{2}),
\end{equation}
and
\begin{equation}
\label{eqn:l2-sol-visible}
\ker(\eHesse^*)
\cap
\ker(\eGrad^*)
\cap
\ker(\eXRT{2})
=
0.
\end{equation}

First, let $f \in \ltwo{2}$ be so that $\eXRT{2}f = 0$. By Theorem~\ref{thm:helmholtz} there are unique $P \in \pot_2$ and $S \in \sol_2$ so that $f = P + S$. Then by~\eqref{eqn:l2-pot-invisible} we have $\eXRT{2}S = 0$ since $\eXRT{2}f = 0$. Thus by~\eqref{eqn:l2-sol-visible} we find that $S = 0$. Hence $f = P = \eGrad W + \eHesse h$ for some $W \in H^1_\eGrad$ and $h \in H^2_\eHesse$.

Conversely, let $f = \eGrad W + \eHesse h$ for some $W \in H^1_\eGrad$ and $h \in H^2_\eHesse$. Then it follows from~\eqref{eqn:l2-pot-invisible} that $\eXRT{2}f = 0$.
\end{proof}

\begin{remark}
If we denote $M(f) \coloneqq fI$ where $f$ is a function and $I$ is the identity matrix then it holds that $K\circ \nabla = H \circ M$. Moreover, we have that
\begin{equation}
\im(\eHesse)\cap \im(\eGrad) = \im(\eGrad\circ \nabla) = \im(\eHesse \circ M),
\end{equation}
as can be verified by taking a Fourier transform.
We do not, however, need these properties.
Therefore the decomposition of a tensor field in the kernel of $\eXRT{2}$ into the images of $\eHesse$ and $\eGrad$ is not unique.
\end{remark}

\section{Computer-assisted proofs}
\label{sec:cas-proofs}

Some of our proofs are partially computer-assisted, and we collect them all in this section due to their similarity.
Computers are only used to set up and solve finite but relatively large linear systems.
These computations are easy to carry out by hand in 2D and 3D, but in higher dimensions a manual approach is too prone to errors.

The computer algebra system we use is the freely available Maxima~\cite{maxima}, and our code is available alongside the arXiv submission~\cite{this}.

\begin{proof}[Proof of Lemma~\ref{lma:no-pw-kernel}]
First, let us argue why proving the lemma for $n\leq4$ implies it for all dimensions.
Suppose that $n>4$ and~\eqref{eq:lma:no-pw-kernel} holds.
We want to show that all components of $a$ vanish, so pick any quadruplet $i_0j_0k_0l_0$ of indices.
Let $W\subset\R^n$ be the subspace spanned by the basis vectors $e_{i_0},e_{j_0},e_{k_0},e_{l_0}$.
If two or more of the four indices coincide, add basis vectors so that $\dim(W)=4$.
By assumption~\eqref{eq:lma:no-pw-kernel} holds for all unit vectors $v\in\sphere{n-1}$ and $q\in\polarizationset{v}{}$, so it also holds with the added assumption that $v\in W$ and $q\in W$.
This is exactly the setting of the lemma when $n=4$, so we conclude that $a_{i_0j_0k_0l_0}=0$.

Let us then turn to the low-dimensional cases.
The case $n=1$ is trivial (the space $\stiffnesstensor{1}$ is one-dimensional), so we are left with $2\leq n\leq4$.
We describe the case $n=3$ here in some detail.
The other two dimensions are analogous and the details can be found in the accompanying Maxima files.

The condition we assume in~\eqref{eq:lma:no-pw-kernel} is linear in $a\in\stiffnesstensor{3}$.
There are infinitely many conditions, and we turn them into an equivalent finite set.
First of all, the conditions are homogeneous in $v$ and $q$, so their normalization is irrelevant.
We parametrize the vector $v$ as
\begin{equation}
\label{eq:cas-3D-v}
v
=
(x,y,1)
\in
\R^3
\end{equation}
with $x,y\in\R^2$.
Upon normalization this covers the unit sphere apart from the equator $z=0$.
Information on the equator is contained in the behavior when $x$ or $y$ tends to infinity.
It may be helpful to think of $(x,y)$ as an element of the projective plane rather than the affine plane.

The polarization vector $q$ can be of two different kinds.
Either it is parallel to $v$ (corresponding to P-waves) or it is orthogonal to $v$ (corresponding to S-waves).
In the former case we take $q=v$, again without loss of generality due to homogeneity.
In the latter case we take
\begin{equation}
q
=
(y,-x+r,-ry)
\in\R^3,
\end{equation}
parametrized with $x,y,r\in\R^3$.
Up to normalization and taking suitable limits at infinity, this parametrizes all vectors orthogonal to the $v$ of~\eqref{eq:cas-3D-v}.

The assumption~\eqref{eq:lma:no-pw-kernel} is now split to two cases.
In the P-wave case the condition is a polynomial in $(x,y)$ and in the S-wave case it is a polynomial in $(x,y,r)$.
The assumption~\eqref{eq:lma:no-pw-kernel} is equivalent with both of these polynomials functions being identically zero.
This, in turn, is equivalent with the coefficients of both polynomials being all zero.
(The coefficients of maximal degree in $(x,y)$ correspond to taking $v$ to be in the equatorial plane $z=0$, so that information was indeed not lost.)
Each coefficient is a linear combination of components of $a\in\stiffnesstensor{3}$.

What we do on Maxima is to form these polynomials, create the list of coefficients, and solve the resulting linear system for $a$.
The only solution is indeed $a=0$.

The only difference in dimensions $n=2$ and $n=4$ is the parametrization of the velocity $v$ (requiring $n-1$ parameters) and the S-wave polarization $q$ (requiring $n-2$ new parameters).
\end{proof}

The benefit of parametrizing our vectors as polynomial rather than trigonometric functions of parameters is that extracting an equivalent set of conditions is far more straightforward and something we may easily trust a computer with.
Trigonometric identities complicate matters; for a 1D example, assuming $a+b\sin^2(\theta)+c\cos^2(\theta)=0$ for all $\theta\in\R$ only implies that $a+c=b-c=0$.
The polynomial approach only gives the slight discomfort of not having the relevant vector normalized and some of the information being pushed to infinity.

\begin{proof}[Proof of Lemma~\ref{lma:reduction}]
Arguing as in the proof of lemma~\ref{lma:no-pw-kernel}, it is enough to prove the claim for dimensions $2$, $3$, $4$ and $5$.

We describe the case $n = 3$ in some detail. The other dimensions are analogous and the details can be found in the accompanying Maxima files.

Assume that $a_{11ij}(p) = 0$ and that~\eqref{eq:vv1} and~\eqref{eq:vv2} hold. We proceed analogous to the proof of lemma~\ref{lma:no-pw-kernel}. The conditions are again homogeneous in $v$ and $q$, and we parametrize the vector $v$, which is now perpendicular to $p$, by
\begin{equation}
v
=
(0,1,x)
\in \R^3
\end{equation}
with $x \in \R$. The polarization vector $q$ can either be parallel to $v$ or orthogonal to $v$. In the former case we take $q = v$ and in the latter we take
\begin{equation}
q
=
(y,x,-1)
\in
\R^3
\end{equation}
with $x,y \in \R$. Given these parametrizations the conditions on $a$ are equivalent to certain polynomials in $x$ and $y$ vanishing identically. The main difference to the proof of lemma~\ref{lma:no-pw-kernel} is the set of vectors which $v$ belongs to. This is reflected in the parametrization.

We create these polynomials in Maxima, extract the lists of coefficients, and solve the resulting linear system for $a$. The result is~\eqref{eqn:reduction1}.

Conversely, proving that~\eqref{eqn:reduction1} together with $a_{11ij}(p) = 0$ implies that~\eqref{eq:vv1} and~\eqref{eq:vv2} hold is easily checked by substitution. Thus the conditions are equivalent as claimed.
\end{proof}

\appendix

\section{Spectral perturbation theory}
\label{app:spt}

This appendix is a soft overview of the basic phenomena of spectral perturbations of real symmetric matrices relevant for the linearization discussion in Section~\ref{sec:linearization}.
For more details, see~\cite{Kato1966}.
A similar discussion to ours is found in many textbooks of quantum mechanics.

Let $A(s)$, $s\in(-\eps,\eps)$, be a family of real symmetric matrices.
Suppose there exists a family of orthonormal eigenvectors $v_k(s)$ and eigevalues $\lambda_k(s)$ depending $C^1$-smoothly on $s$.\footnote{%
We take the existence of such smooth families as our starting point. For a proof that there are holomorphic $\lambda_k(\dummy)$ and $v_k(\dummy)$ when $A(\dummy)$ is holomorphic, see \cite[Chapter 2, Theorem 1.10]{Kato1966}.
For the $C^1$ case (only for eigenvalues), see Theorem 6.8 in the same chapter.
}
The eigenvalues $\lambda_k(0)$ need not be distinct.
When we omit the parameter $s$, we take it to be zero, and we denote the derivative in $s$ by prime.

Differentiating the eigenvalue equation yields
\begin{equation}
\label{eq:spt1}
A'v_k+Av'_k
=
\lambda'_kv_k+\lambda_kv_k'
\end{equation}
and differentiating the normalization property yields
\begin{equation}
v'_k\cdot v_k
=
0
.
\end{equation}
Taking the inner product of $v_k$ and~\eqref{eq:spt1} gives
\begin{equation}
\label{eq:spt3}
\lambda'_k
=
v_k\cdot Av_k
.
\end{equation}
When the eigenvalue $\lambda_k$ is simple, the vector $v_k$ is uniquely defined (up to sign).
When the eigenvalue degenerates, we may not choose an eigenbasis of each eigenspace as we like, but it is determined by the perturbation $A'$ (and in some cases higher order derivatives).

Taking the inner product of $v_l$ and~\eqref{eq:spt1} gives
\begin{equation}
\label{eq:spt2}
v_l\cdot A'v_k
=
(\lambda_k-\lambda_l)v_l\cdot v'_k
.
\end{equation}
Let $P_k$ be the orthogonal projection to the eigenspace of $\lambda_k$, thought of as a non-square matrix.
The adjoint $P_k^T$ is the inclusion map.

The block of the perturbation $A'$ corresponding to the eigenspace of $\lambda_k$ is
\begin{equation}
A'_k
=
P_k A' P_k^T
.
\end{equation}
When $l\neq k$ but $\lambda_l=\lambda_k$, equation~\eqref{eq:spt2} shows that $A'_kv_k$ is orthogonal to $v_l$.
This is true for all such indices $l\neq k$, so for dimensional reasons $v_k$ is an eigenvector of $A'_k$.
By~\eqref{eq:spt3} the associated eigenvalue is $\lambda'_k$.

Therefore the first order perturbations to the spectrum of $A$ are the eigenvalues of the perturbation $A'$ restricted to the associated eigenblocks of $A$.
Because the map from a matrix to its eigenvalues is not linear, this ``spectral derivative'' is also not linear.
The spectrum of a matrix is Gateaux differentiable but not Fr\'echet differentiable if the perturbation breaks degeneracies.

For a concrete example, take the background matrix
\begin{equation}
A
=
\begin{pmatrix}
3&0&0\\
0&1&0\\
0&0&1
\end{pmatrix}
\end{equation}
and its first order perturbation
\begin{equation}
A'
=
\begin{pmatrix}
-2&0&5\\
0&1&2\\
5&2&1
\end{pmatrix}
.
\end{equation}
The first eigenvalue is $\lambda_1=3$, and its $s$-derivative is $\lambda_1'=-2$.
For the degenerate eigenvalue $\lambda_2=\lambda_3=1$, we first have to identify the block
\begin{equation}
A'_2
=
A'_3
=
\begin{pmatrix}
1&2\\
2&1
\end{pmatrix}
.
\end{equation}
Its eigenvalues are $-1$ and $3$, so $\lambda'_2=-1$ and $\lambda'_3=3$, or vice versa.

In our model in three or more dimensions one of the two eigenspaces of the background Christoffel matrix is degenerate.
While we should only use the spectrum of the corresponding perturbation block, we use the whole block instead.
This turns the linearized problem linear.

\bibliographystyle{abbrv}
\bibliography{bibliography}

\end{document}